\documentclass[11pt]{article}
\usepackage{fullpage,url,natbib}
\usepackage{times}

\usepackage{amsthm,amsfonts,amsmath,amssymb,epsfig,color,float,graphicx,verbatim}
\usepackage{enumitem}
\usepackage{algorithm}
\usepackage{algpseudocode}

\usepackage{footmisc}
\usepackage{wrapfig}
\usepackage{subcaption}
\usepackage{bbm, bm}
\usepackage{natbib}

\usepackage{hyperref}
\hypersetup{
	colorlinks   = true, 
	urlcolor     = blue, 
	linkcolor    = blue, 
	citecolor   = blue 
}

\newtheorem{theorem}{Theorem}
\newtheorem*{theorem*}{Theorem}
\newtheorem{proposition}[theorem]{Proposition}
\newtheorem*{proposition*}{Proposition}

\newtheorem{lemma}[theorem]{Lemma}

\newtheorem{definition}[theorem]{Definition}

\newtheorem{assumption}[theorem]{Assumption}

\renewcommand{\eqref}[1]{Eq.~(\ref{#1})}

\newcommand{\thmref}[1]{Theorem~\ref{#1}}

\renewcommand{\S}{\mathbb{S}}

\newcommand{\out}{\mathrm{out}}

\newcommand{\conv}{\mathrm{conv}}

\newcommand{\Unif}{\mathrm{Unif}}

\newcommand{\reals}{\mathbb{R}}
\newcommand{\B}{\mathbb{B}}

\newcommand{\bpar}{\bar{\partial}}

\newcommand{\vertiii}[1]{{\left\vert\kern-0.25ex\left\vert\kern-0.25ex\left\vert #1 
    \right\vert\kern-0.25ex\right\vert\kern-0.25ex\right\vert}}

\usepackage{xcolor}

\newcommand{\beq}{\begin{eqnarray*}}
\newcommand{\eeq}{\end{eqnarray*}}
\newcommand{\beqn}{\begin{eqnarray}}
\newcommand{\eeqn}{\end{eqnarray}}

\usepackage{ifmtarg}
\usepackage{xifthen}%
\newcommand{\ent}[1][]{%
\ifthenelse{\isempty{#1}}{%
\mathrm{H}
}{
\mathrm{H}^{(#1)}
}}

\newcommand{\loch}[1][]{%
\ifthenelse{\isempty{#1}}{%
\mathrm{h}
}{
\mathrm{h}^{(#1)}
}}

\newcommand{\NN}{\mathbb{N}}

\newcommand{\E}{\mathbb{E}}

\newcommand{\bx}{\mathbf{x}}
\newcommand{\bw}{\mathbf{w}}
\newcommand{\bg}{\mathbf{g}}

\newcommand{\bz}{\mathbf{z}}

\newcommand{\by}{\mathbf{y}}

\newcommand{\norm}[1]{\left\|#1\right\|}
\newcommand{\inner}[1]{\langle#1\rangle}

\title{
An Algorithm with Optimal Dimension-Dependence for Zero-Order Nonsmooth Nonconvex Stochastic Optimization
}

\author{
Guy Kornowski
\qquad
Ohad Shamir
\vspace{3pt}
\\
Weizmann Institute of Science \\
\texttt{\{guy.kornowski,ohad.shamir\}@weizmann.ac.il}  
}

\begin{document}

\maketitle

\begin{abstract}
We study the complexity of producing $(\delta,\epsilon)$-stationary points of Lipschitz objectives which are possibly neither smooth nor convex, using only noisy function evaluations.
Recent works proposed several stochastic zero-order algorithms that solve this task, all of which suffer from a dimension-dependence of $\Omega(d^{3/2})$ where $d$ is the dimension of the problem, which was conjectured to be optimal. We refute this conjecture by providing a faster algorithm that has complexity $O(d\delta^{-1}\epsilon^{-3})$, which is optimal (up to numerical constants) with respect to $d$ and also optimal with respect to the accuracy parameters $\delta,\epsilon$, thus solving an open question due to \citet{lin2022gradient}. Moreover, the convergence rate achieved by our algorithm
is also optimal for smooth objectives,
proving that in the nonconvex stochastic zero-order setting, \emph{nonsmooth optimization is as easy as smooth optimization}. 
We provide algorithms that achieve the aforementioned convergence rate in expectation as well as with high probability.
Our analysis is based on a simple yet powerful lemma regarding the Goldstein-subdifferential set,
which allows utilizing recent advancements in first-order nonsmooth nonconvex optimization.
\end{abstract}

\section{Introduction}

We consider the problem of optimizing a stochastic objective of the form
\[
f(\bx)=\E_{\xi\sim\Xi}[F(\bx;\xi)]
\]
where the stochastic components $F(\,\cdot\,;\xi):\reals^d\to\reals$
are Lipschitz continuous,
yet possibly not smooth nor convex. We consider 
stochastic
zero-order (also known as \emph{gradient-free} or \emph{derivative-free}) algorithms that have access only to noisy function evaluations. At each time step, the algorithm draws $\xi\sim\Xi$ and can observe $F(\bx,\xi)$ for points $\bx\in\reals^d$ of its choice.
Problems of this type arise throughout machine learning, control theory and finance, in applications in which gradients are expensive (or even impossible) to evaluate, see for example the book by \citet{spall2005introduction} for an overview.
Although in the convex setting the complexity of such algorithms is relatively well understood \citep{agarwal2010optimal,duchi2015optimal,nesterov2017random,shamir2017optimal}, much less is known about the nonsmooth nonconvex setting. This challenging setting is of major interest in modern deep learning applications, where objective functions of interest are typically neither smooth or convex. For example, stochastic zero-order optimization methods were applied to fine-tune large language models \citep{malladi2023fine}.

Recently, \citet{lin2022gradient} proposed a gradient-free algorithm that produces a $(\delta,\epsilon)$-stationary point using $O(d^{3/2}\delta^{-1}\epsilon^{-4})$ function evaluations.
Following \citet{Zhang-2020-Complexity}, recall that a point $\bx$ is called a $(\delta,\epsilon)$-stationary point if there exists a convex combination of gradients in a $\delta$-neighborhood of $\bx$ whose norm is less than $\epsilon$ (see Section~\ref{sec: prelim} for a reminder of relevant definitions).
\citet{lin2022gradient} posed the question as to whether this super-linear dimension dependence is inevitable or not.
The aforementioned complexity was very recently improved to $O(d^{3/2}\delta^{-1}\epsilon^{-3})$ by \citet{chen2023faster}, yet notably, this result still suffers from the same super-linear dimension dependence.

In particular, as pointed out by \citeauthor{lin2022gradient}, this dimension dependence is $\Omega(\sqrt{d})$ worse than that of stochastic zero-order \emph{smooth} nonconvex optimization, a setting in which it is possible to find an $\epsilon$-stationary point (i.e. $\bx$ such that $\norm{\nabla f(\bx)}\leq\epsilon$)
using $O(d\epsilon^{-4})$ noisy function evaluations \citep{ghadimi2013stochastic}. This led the authors to conjecture that stochastic zero-order nonsmooth nonconvex optimization is
``likely to be intrinsically harder''
than its smooth counterpart.

Our main contribution resolves this open question, showing that this is actually \emph{not} the case. We propose a
faster
zero-order algorithm for nonsmooth nonconvex optimization, which requires only $O(d\delta^{-1}\epsilon^{-3})$ noisy function evaluations. As we will soon argue, this complexity has an optimal linear-dependence on the dimension $d$, while also obtaining the optimal dependence with respect to the accuracy parameters $\delta$ and $\epsilon$.
All of these dependencies are known to be optimal even if $f(\cdot)$ is smooth, implying that in the stochastic zero-order setting, \emph{nonsmooth nonconvex optimization is as easy as smooth nonconvex optimization}.
Moreover, when the objective is smooth, our proposed algorithm automatically recovers the $O(d\epsilon^{-4})$ complexity of stochastic gradient-free smooth nonconvex optimization \citep{ghadimi2013stochastic}.
Whether this adaptivity property is possible was originally raised as an open question by \citet{Zhang-2020-Complexity} in the context of first-order algorithms
(that have access to gradient information), and was recently 
confirmed
by \citet{cutkosky2023optimal}. Our result extends the resolution of this question to the case of zero-order algorithms.

As previously mentioned, the dependencies on $d,\delta$ and $\epsilon$ we obtain are all optimal. Indeed,
the linear dimension dependence is well-known to be inevitable for gradient-free algorithms even in the strictly-easier cases of smooth or convex optimization \citep{duchi2015optimal}, while the implied $\epsilon^{-4}$ factor is known to be inevitable even in the strictly-easier case of stochastic first-order smooth optimization with exact function evaluations \citep{arjevani2023lower}.\footnote{Technically, the lower bound construction in \citet{arjevani2023lower} is not globally Lipschitz, yet a slight modification of it which appears in \citet[Appendix F]{cutkosky2023optimal} is.}
Interestingly, in terms of the dependence on $\delta$ and $\epsilon$, the convergence rate we obtain is as fast as the currently best-known \emph{deterministic first-order} algorithms for nonsmooth nonconvex optimization \citep{Zhang-2020-Complexity,Tian-2022-Finite,Davis-2022-Gradient}. This is in stark contrast to smooth nonconvex optimization, in which optimal stochastic and deterministic methods have disparate complexities on the order of $\epsilon^{-4}$ and $\epsilon^{-2}$, respectively \citep{arjevani2023lower,carmon2020lower}.
We also note that our algorithm is a factor of $\Omega(\sqrt{d})$ faster even than the previously best-known rate for deterministic (i.e. noiseless, when $\Xi=\{\xi\}$) zero-order nonsmooth nonconvex optimization, and that it also obtains an improved dependence on the Lipschitz parameter when compared to the previously mentioned works in this setting.
Finally, we remark that while the dependence on each parameter by itself (i.e. $d,\delta$ and $\epsilon$) is already known to be optimal, currently there is no result in the literature formally proving a lower bound \emph{jointly} in these parameters, namely of the form $\Omega(d\delta^{-1}\epsilon^{-3})$, which will automatically follow from a smooth (joint) lower bound of the form $\Omega(d\epsilon^{-4})$.
We conjecture these results should be obtainable by modifying the analysis of \citet{arjevani2023lower} to incorporate zero-order oracles.

\section{Preliminaries.} \label{sec: prelim}

\paragraph{Notation.}
We use bold-faced font to denote vectors, e.g. $\bx\in\reals^d$, and denote by $\norm{\bx}$ the Euclidean norm.
We denote by $[n]:=\{1,\dots,n\},~\B(\bx,\delta):=\{\by\in\reals^d:\norm{\by-\bx}\leq\delta\}$, and by $\S^{d-1}\subset\reals^d$ the unit sphere.
We denote by $\textnormal{conv}(\cdot)$ the convex hull operator, and by $\mathrm{Unif}(A)$ the uniform measure over a set $A$.
We use the standard big-O notation, with $O(\cdot)$, $\Theta(\cdot)$ and $\Omega(\cdot)$ hiding absolute constants that do not depend on problem parameters, $\tilde{O}(\cdot)$ and $\tilde{\Omega}(\cdot)$ hiding absolute constants and additional logarithmic factors.

\paragraph{Nonsmooth analysis.}
We call a function $f:\reals^d\to\reals$ $L$-Lipschitz if for any $\bx,\by\in\reals^d:|f(\bx)-f(\by)|\leq L\norm{\bx-\by}$, and $H$-smooth if it is differentiable and $\nabla f:\reals^d\to\reals^d$ is $H$-Lipschitz, namely for any $\bx,\by\in\reals^d:\norm{\nabla f(\bx)-\nabla f(\by)}\leq H\norm{\bx-\by}$.
By Rademacher's theorem, Lipschitz functions are differentiable almost everywhere (in the sense of Lebesgue). Hence, for any Lipschitz function $f:\reals^d\to\reals$ and point $\bx\in\reals^d$ the Clarke subgradient set \citep{Clarke-1990-Optimization} can be defined as
\[
\partial f(\bx):=\conv\{\bg\,:\,\bg=\lim_{n\to\infty}\nabla f(\bx_n),\,\bx_n\to \bx\}~,
\]
namely, the convex hull of all limit points of $\nabla f(\bx_n)$ over all sequences of differentiable points which converge to $\bx$.
Note that if the function is continuously differentiable at a point or convex, the Clarke subdifferential reduces to the gradient or subgradient in the convex analytic sense, respectively.
We say that a point $\bx$ is an $\epsilon$-stationary point of $f(\cdot)$ if $\min\{\norm{\bg}:\bg\in\partial f(\bx)\}\leq\epsilon$.
Furthermore, given $\delta\geq0$ the Goldstein $\delta$-subdifferential \citep{Goldstein-1977-Optimization} of $f$ at $\bx$ is the set
\[
\partial_{\delta}f(\bx):=\conv\left(\cup_{\by\in\B(\bx,\delta)}\partial f(\by)\right)~,
\]
namely all convex combinations of gradients at points in a $\delta$-neighborhood of $\bx$.
We denote the minimum-norm element of the Goldstein $\delta$-subdifferential by
\[
\bpar_{\delta}f(\bx):={\arg\min}_{\bg\in\partial_{\delta} f(\bx)}\norm{\bg}~.
\]

\begin{definition}
A point $\bx$ is called a $(\delta,\epsilon)$-stationary point of $f(\cdot)$ if $\norm{\bpar_{\delta}f(\bx)}\leq\epsilon$.
\end{definition}

Note that a point is $\epsilon$-stationary if and only if it is $(\delta,\epsilon)$-stationary for all $\delta\geq 0$ \citep[Lemma 7]{Zhang-2020-Complexity}.
Moreover, if $f$ is $H$-smooth and $\bx$ is a $(\frac{\epsilon}{3H},\frac{\epsilon}{3})$-stationary point of $f$, then it is also $\epsilon$-stationary \citep[Proposition 6]{Zhang-2020-Complexity}.

\paragraph{Randomized smoothing.}

Given a Lipschitz function $f:\reals^d\to\reals$, we define its uniform smoothing
\[
f_\rho(\bx):=\E_{\bz\sim\Unif(\B(\mathbf{0},1))}[f(\bx+\rho\bz)]~.
\]
It is well known (cf. \citealp{yousefian2012stochastic}) that if $f$ is $L_0$-Lipschitz, then 
\begin{itemize}
    \item $f_\rho$ is $L_0$-Lipschitz;
    \item $f_\rho$ is $O(\sqrt{d}L_0\rho^{-1})$-smooth;
    \item $|f(\bx)-f_{\delta}(\bx)|\leq \rho L_0$ for all $\bx\in\reals^d$.
\end{itemize}

\paragraph{Setting.}

We consider optimization objectives of the form $f(\bx)=\E_{\xi\sim\Xi}[F(\bx;\xi)]$, where $\xi\sim\Xi$ is a random variable.
We impose the assumption that the stochastic components $F(\,\cdot\,;\xi):\reals^d\to\reals$ are Lipschitz continuous, possibly with a  varying Lipschitz constant:

\begin{assumption} \label{ass: Lip}
For any $\xi$, the function $F(\,\cdot\,;\xi)$ is $L(\xi)$-Lipschitz. Moreover, we assume $L(\xi)$ has a bounded second moment: Namely, there exists $L_0>0$ such that 
\[
\E_{\xi\sim\Xi}[L(\xi)^2]\leq L_0^2~.
\]
\end{assumption}

We note that Assumption~\ref{ass: Lip} is weaker than assuming $F(\,\cdot\,;\xi)$ is $L_0$-Lipschitz for all $\xi$. We also remark that in case $\Xi$ is supported on a single point
then the optimization problem reduces to that of a $L_0$-Lipschitz objective using exact evaluations.

\section{Algorithms and Main Results} \label{sec: main}

Before formally presenting our main result, we find it insightful to stress out the key idea, and in particular how our algorithm differs than those of \citet{lin2022gradient,chen2023faster}.
The main strategy employed by both of these papers is based on the following result.

\begin{proposition}[\citealp{lin2022gradient}, Theorem 3.1] \label{thm: Lin's randomized smoothing}
For any $\rho\geq 0:\nabla f_\rho(\bx)\in\partial_\rho f(\bx)$. Hence, for $\rho=\delta$, if $\bx$ is an $\epsilon$-stationary point of $f_\delta$, then it is a $(\delta,\epsilon)$-stationary point of $f$. 
\end{proposition}

Following this observation, both papers set out to design algorithms that produce an $\epsilon$-stationary point of $f_\delta$. A well known technique (which we formally recall later on) allows to use two possibly noisy evaluations of $f$ in order to produce a stochastic \emph{first-order} oracle of $f_\delta$ whose second moment is bounded by $\sigma^2=O(d)$.
Noting that $f_\delta$ is $L_1=O(\sqrt{d}/\delta)$-smooth, the standard analysis of stochastic gradient descent (SGD) for smooth nonconvex optimization shows that it obtains an $\epsilon$-stationary point of $f_\delta$ within $O(\sigma^2 L_1\epsilon^{-4})=O(d^{3/2}\delta^{-1}\epsilon^{-4})$ oracle calls, recovering the main result of \citet{lin2022gradient}. The improved $\epsilon$-dependence due to \citet{chen2023faster} was achieved by employing a variance-reduction method instead of plain SGD, though other than that, their main algorithmic strategy and analysis are the same.

Moreover, the algorithmic strategy we have described seems to reveal a barrier, (mistakenly) suggesting the $d^{3/2}$ dependence is unavoidable. Indeed, it is relatively straightforward to see that any
gradient estimator which is based on a constant number of function evaluations
must have variance of at least $\sigma^2=\Omega(d)$, 
while it is also known that any efficient smoothing technique must suffer from a smoothness parameter of at least $L_1=\Omega(\sqrt{d})$ \citep{kornowski2022oracle}. Since the complexity of any stochastic first-order method for smooth nonconvex optimization must scale at least as $\Omega(\sigma^2 L_1)$ \citep{arjevani2023lower} which in this case is unavoidably $\Omega(d^{3/2})$, we are stuck with this factor.

The main technical ingredient that allows us to reduce this factor is the following result (to be proved in Section~\ref{sec: main proof}), which examines the Goldstein $\delta$-subdifferential set under randomized smoothing.

\begin{lemma} \label{lem: delta,eps of smooth}
For any $\rho,\nu\geq 0:\partial_{\nu} f_{\rho}(\bx)\subseteq\partial_{\rho+\nu} f(\bx)$. Hence, if $\bx$ is an $(\nu,\epsilon)$-stationary point of $f_\rho$, then it is a $(\rho+\nu,\epsilon)$-stationary point of $f$.
In particular, as long as $\rho+\nu\leq\delta$ it is a $(\delta,\epsilon)$-stationary point of $f$.
\end{lemma}

Note that the lemma above strictly generalizes 
Proposition~\ref{thm: Lin's randomized smoothing} \citep[Theorem 3.1]{lin2022gradient}
which is readily recovered by plugging $\nu=0$. The utility of this result is that it allows to replace the task of finding an $\epsilon$-stationary point of $f_{\delta}$ to that of finding a $(\delta,\epsilon)$-stationary point of it (disregarding a constant factor multiplying $\delta$).
To see why this is beneficial, recall that while $f_{\delta}$ is $O(\sqrt{d}L_{0}\delta^{-1})$-smooth, it is merely $L_{0}$-Lipschitz! Thus using a stochastic first-order \emph{nonsmooth} nonconvex algorithm which scales with the Lipschitz parameter (instead of the smoothness parameter), we save a whole $\Omega(\sqrt{d})$ factor, yielding the optimal dimension dependence.
In particular, using the optimal stochastic first-order algorithm of \citet{cutkosky2023optimal} 
that has complexity $O(\sigma^2\delta^{-1}\epsilon^{-3})$, as described in Algorithm~\ref{alg: main}, results in the following convergence guarantee:\footnote{It is interesting to note that
using the stochastic algorithm of \citet{Zhang-2020-Complexity} (instead of Algorithm~\ref{alg: main}), when paired with our analysis, yields the desired linear dimension dependence as well -- albeit with with a
worse
convergence rate with respect to $\epsilon$, on the order of $d\delta^{-1}\epsilon^{-4}$.}

\begin{theorem} \label{thm: upper}
Let $\delta,\epsilon\in(0,1)$, and suppose $f(\bx_0)-\inf_{\bx}f(\bx)\leq\Delta$.
Under Assumption~\ref{ass: Lip}, there exists
\[
T
=O\left(\frac{d L_0^2\Delta}{\delta\epsilon^3}\right)
\]
such that setting
\[
\rho=\min\left\{\frac{\delta}{2},\frac{\Delta}{L_0}\right\},~\nu=\max\left\{\frac{\delta}{2},\delta-\frac{\Delta}{L_0}\right\},
~D=
\left(\frac{(\Delta+\rho L_0)\sqrt{\nu}}{\sqrt{d}L_0 T}\right)^{2/3},
~\eta=\frac{\Delta+\rho L_0}{dL_0^2 T},
\]
and running
Algorithm~\ref{alg: main} with 
Algorithm~\ref{alg: grad}
as a subroutine,
outputs a point $\bx^{\out}$ satisfying
$
\E\left[\norm{\bpar_{\delta}f(\bx^{\out})}\right]\leq\epsilon$ using $2T$ noisy function evaluations.
\end{theorem}

\begin{algorithm}[H]
\begin{algorithmic}[1]\caption{Optimal Stochastic Nonsmooth Nonconvex Optimization Algorithm}\label{alg: main}
\State \textbf{Input:} Initialization $\bx_0\in\reals^d$, smoothing parameter $\rho>0$, accuracy parameter $\nu>0$, clipping parameter $D>0$, step size $\eta>0$, iteration budget $T\in\NN$.
\State \textbf{Initialize:} $\bm{\Delta}_1=\mathbf{0}$
\For{$t=1,\dots,T$}
\State Sample $\xi_t\sim\Xi$
\State Sample $s_t\sim\Unif[0,1]$
\State $\bx_t=\bx_{t-1}+\bm{\Delta}_t$
\State $\bz_t=\bx_{t-1}+s_t\bm{\Delta}_t$
\State $\bg_t=\textsc{GradEstimator}(\bz_t,\rho,\xi_t)$
\Comment{\texttt{
Uses two noisy function evaluations}}
\State $\bm{\Delta}_{t+1}
=\min\left(1,\frac{D}{\norm{\bm{\Delta}_t-\eta\bg_t}}\right)\cdot\left(\bm{\Delta}_t-\eta\bg_t\right)$
\EndFor
\State $M=\lfloor\frac{\nu}{D}\rfloor,~K=\lfloor\frac{T}{M}\rfloor$
\For{$k=1,\dots,K$}
\State $\overline{\bx}_{k}=\frac{1}{M}\sum_{m=1}^{M}\bz_{(k-1)M+m}$
\EndFor
\State $k_{\out}\sim \Unif\{1,\dots,K\}$
\State $\bx^{\out}=\overline{\bx}_{k_{\out}}$
\State \textbf{Output:} $\bx^{\out},~(\bz_{(k_{\out}-1)M+m})_{m=1}^{M}$. 
\end{algorithmic}
\end{algorithm}

\begin{algorithm}[H]
\begin{algorithmic}[1]\caption{$\textsc{GradEstimator}(\bx,\rho,\xi)$}\label{alg: grad}
\State \textbf{Input:} Point $\bx\in\reals^d$, smoothing parameter $\rho>0$, random seed $\xi$.
\State Sample $\bw\sim\Unif(\S^{d-1})$
\State Evaluate $F(\bx+\rho\bw;\xi)$ and $F(\bx-\rho\bw;\xi)$
\State $\bg=\tfrac{d}{2\rho}(F(\bx+\rho\bw;\xi)-F(\bx-\rho\bw;\xi))\bw$
\Comment{\texttt{Unbiased estimator of }$\nabla f_{\rho}(\bx)$}
\State \textbf{Output:} $\bg$. 
\end{algorithmic}
\end{algorithm}

\paragraph{Parallel complexity.}
At each iteration, Algorithm~\ref{alg: main} determines $\bg_t$ by calling $\textsc{GradEstimator}$ (Algorithm~\ref{alg: grad}), which requires $2$ evaluations of $F(\,\cdot\,;\xi_t)$. More generally, $\bg_t$ can be set as the average of $k$ independent, possibly parallel calls to this subroutine, which would require $2k$ function evaluations. By an easy generalization of the second-order moment bound which we present in Lemma~\ref{lem: estimator} (as part of the proof of Theorem~\ref{thm: upper}), this averaging would decrease the second-moment on the order of
    \[
    \E[\norm{\bg_t}^2|\bx_t,s_t,\bm{\Delta}_t]\lesssim\frac{dL_0^2}{k}+\norm{\E[\bg_t]}^2
    \leq
    L_0^2\left(\frac{d}{k}+1\right)~.
    \]
    With the rest of the proof of Theorem~\ref{thm: upper} as is, this yields an expected number of rounds of
    \[
    T
    =O\left(\left(\frac{d}{k}+1\right)\cdot\frac{\Delta L_0^2}{\delta\epsilon^3}\right)~,
    \]
    though the total number of queries would be $k$ times larger than above, and equal to 
    \[
    O\left(\left(d+k\right)\cdot\frac{\Delta L_0^2}{\delta\epsilon^3}\right)~.
    \]
    In particular, letting $k=\Theta(d)$ removes the dimension dependence in the parallel complexity altogether, while maintaining the same complexity overall (up to a constant). Notably, this even matches the currently best-known complexity for deterministic first-order algorithms under the discussed setting \citep{Zhang-2020-Complexity,Tian-2022-Finite,Davis-2022-Gradient}.

The trick of averaging gradient estimators in order to reduce the dimension-factor in the parallel complexity is applicable to smooth or convex optimization as well,
though under those settings the resulting overall complexity is still significantly worse (in terms of the $\epsilon$-dependence) than optimal deterministic first-order methods, as mentioned in the introduction (cf. \citealp{duchi2018minimax,bubeck2019complexity}).

\paragraph{High probability guarantee.}

While Theorem~\ref{alg: main} shows that Algorithm~\ref{alg: main} yields the desired expected complexity, many practical applications require high probability bounds, namely producing a point $\bx$ such that 
\[
\Pr\left[\norm{\bpar_{\delta} f(\bx)}\leq\epsilon\right]\geq1-\gamma
\]
for some small $\gamma>0$. A naive application of Markov's inequality to the expected complexity shows that Algorithm~\ref{alg: main} produces such a point within
\begin{equation} \label{eq: markov}
T=O\left(\frac{d L_0^2\Delta}{\delta\epsilon^3\gamma^3}\right)
\end{equation}
noisy function evaluations, which is rather crude with respect to the probability parameter $\gamma$. Adapting a technique due to \citet{ghadimi2013stochastic} to our setting, we can design an algorithm with a significantly tighter high-probability bound.
The original idea of \citet{ghadimi2013stochastic} for the case of smooth stochastic optimization, which was also used by \citet{lin2022gradient}, consists of several independent calls to the main algorithm, yielding a list of candidate points. Subsequently, a post-optimization phase estimates the gradient norm of any such point, returning the minimal --- which is proved likely to succeed due to a concentration argument.
We note that adapting this technique to our setting is not trivial, since the post-optimization phase should attempt at estimating $\norm{\bpar_{\delta} f(\cdot)}$ rather than $\norm{\nabla f(\cdot)}$, which is hard in general. Luckily, using Lemma~\ref{lem: delta,eps of smooth}, the former can be bounded by $\norm{\bpar_{\nu} f_{\rho}(\cdot)}$, which in turn can be bounded (with high probability) using a sequence of evaluations at nearby points. This procedure is described in Algorithm~\ref{alg: high prob}, whose convergence rate is presented in the following theorem.

\begin{theorem} \label{thm: high prob}
Let $\gamma,\delta,\epsilon\in(0,1)$, and suppose $f(\bx_0)-\inf_{\bx}f(\bx)\leq\Delta$.
Under Assumption~\ref{ass: Lip}, there exist $T=O\left(\frac{d\Delta L_0^2}{\delta\epsilon^3}\right),~R=O(\log(1/\gamma)),~S=O\left(\frac{\log(1/\gamma)}{\gamma}\right)$ 
such that setting
$\rho=\min\left\{\frac{\delta}{2},\frac{\Delta}{L_0}\right\},~\nu=\max\left\{\frac{\delta}{2},\delta-\frac{\Delta}{L_0}\right\},
~D=
\left(\frac{(\Delta+\rho L_0)\sqrt{\nu}}{\sqrt{d}L_0 T}\right)^{2/3},
~\eta=\frac{\Delta+\rho L_0}{dL_0^2 T}$,
and running
Algorithm~\ref{alg: high prob}
outputs a point $\bx^{\out}$ satisfying
\[
\Pr\left[\norm{\bpar_{\delta} f(\bx^{\out})}\leq\epsilon\right]\geq1-\gamma
\]
using 
\[
O\left(\frac{d L_0^2\Delta\log(1/\gamma)}{\delta\epsilon^3}+\frac{d L_{0}^2\log^2(1/\gamma)}{\gamma\epsilon^2}\right)
\]
noisy function evaluations.
\end{theorem}

Notably, the number of function evaluations guaranteed by the theorem above is significantly smaller than in \eqref{eq: markov}. We also remark that even in the easier case in which the function evaluations are noiseless, when relying on a local information (e.g. zero-order or even first-order evaluations), a lack of smoothness or convexity provably necessitates the use of randomization in the optimization algorithm when applied to Lipschitz objectives, thus resorting to high probability guarantees is inevitable \citep{jordan2023deterministic}.

\begin{algorithm}[!t]
\begin{algorithmic}[1]\caption{
Algorithm with Post-Optimization Validation}\label{alg: high prob}
\State \textbf{Input:} Initialization $\bx_0\in\reals^d$, smoothing parameter $\rho>0$, accuracy parameter $\nu>0$, clipping parameter $D>0$, step size $\eta>0$, iteration budget per round $T\in\NN$, number of rounds $R\in\NN$, validation sample size $S\in\NN$.
\State \textbf{Initialize:} $M=\lfloor\frac{\nu}{D}\rfloor$
\For{$r=1,\dots,R$}
\State $\bx_{\out}^r,~(\bz^r_m)_{m=1}^{M}=\text{Algorithm~\ref{alg: main}}(\bx_0,\rho,\nu,D,\eta,T)$
\For{$s=1,\dots,S$}
\For{$m=1,\dots,M$}
\State Sample $\xi_{m,s}\sim\Xi$
\State $\bg^{r}_{m,s}=\textsc{GradEstimator}(\bz^r_m,\rho,\xi_{m,s})$ \Comment{\texttt{Unbiased estimator of }$\nabla f_{\rho}(\bz^r_m)$}
\EndFor
\State $\hat{\bg}^r_s=\frac{1}{M}\sum_{m=1}^{M}\bg^{r}_{m,s}$
\EndFor
\State $\hat{\bg}^r=\frac{1}{S}\sum_{s=1}^{S}\bg^r_{s}$
\EndFor
\State $r^*=\arg\min_{r\in[R]}\norm{\hat{\bg}^r}$
\State \textbf{Output:} $\bx^{r^*}_{\out}$. 
\end{algorithmic}
\end{algorithm}

\section{Proof of \thmref{thm: upper}} \label{sec: main proof}

As previously discussed, the key to obtaining the improved rate is Lemma~\ref{lem: delta,eps of smooth}. We start by proving it, followed by two additional propositions, after which we combine the ingredients in order to conclude the proof.

\begin{proof}[Proof of Lemma~\ref{lem: delta,eps of smooth}]
Let $\bg\in \partial_\nu f_\rho(\bx)$. Then, by definition, there exist $\by_1,\dots,\by_k\in \B(\bx,\nu)$ (for some $k\in\NN$) such that $\bg=\sum_{i\in[k]}\lambda_i \nabla f_\rho(\by_i)$, where $\lambda_1,\dots,\lambda_k\geq0$ with $\sum_{i\in[k]}\lambda_i=1$.
By Proposition~\ref{thm: Lin's randomized smoothing} we have for all $i\in[k]:$ 
\begin{equation} \label{eq: 1st contain}
\nabla f_\rho(\by_i)\in\partial_{\rho} f(\by_i) 
~.
\end{equation}
Further note that since $\norm{\by_i-\bx}\leq\nu$, then by definition
\begin{equation} \label{eq: 2nd contain}
\partial_\rho f(\by_i)\subseteq \partial_{\rho+\nu} f(\bx)~.    
\end{equation}
By combining \eqref{eq: 1st contain} and \eqref{eq: 2nd contain} we get that for all $i\in[k]:~\nabla f_\rho(\by_i)\in\partial_{\rho+\nu}f(\bx)$. Since $\partial_{\rho+\nu}f(\bx)$ is a convex set, we get that 
\[
\bg=\sum_{i\in[k]}\lambda_i \nabla f_\rho(\by_i)
\in\partial_{\rho+\nu}f(\bx)~.
\]
\end{proof}

The following lemma is essentially due to \citet{shamir2017optimal},
showing that it is possible to construct a gradient estimator whose second moment scales linearly with respect to the dimension $d$.

\begin{lemma}
\label{lem: estimator}
Let
\[
\bg_t=\frac{d}{2\rho}\left(F(\bx_t+s_t\bm{\Delta}_t+\rho\bw_t;\xi_t)-F(\bx_t+s_t\bm{\Delta}_t-\rho\bw_t;\xi_t)\right)\bw_t~,
\]
as generated by  $\textsc{GradEstimator}$ (Algorithm~\ref{alg: grad}) when called at iteration $t$ of Algorithm~\ref{alg: main}. Then 
\[
\E_{\xi_t,\bw_t}[\bg_t|\bx_{t-1},s_t,\bm{\Delta}_t]=\nabla f_{\rho}(\bx_{t-1}+s_t\bm{\Delta}_t)=\nabla f_{\rho}(\bz_{t})
\]
and
\[
\E_{\xi_t,\bw_t}[\norm{\bg_t}^2|\bx_{t-1},s_t,\bm{\Delta}_t]\leq16\sqrt{2\pi}dL_{0}^2~.
\]

\end{lemma}

\begin{proof}
For the sake of notational simplicity, we omit the subscript $t$ throughout the proof.
For the first claim, since $-\bw\sim \bw$ are identically distributed, we have
\begin{align*}
\E_{\xi,\bw}[\bg|\bx,s,\bm{\Delta}]
&=\E_{\xi,\bw}\left[\frac{d}{2\rho}\left(F(\bx+s\bm{\Delta}+\rho\bw;\xi)-F(\bx+s\bm{\Delta}-\rho\bw;\xi)\right)\bw\,|\,\bx,s,\bm{\Delta}\right]
\\
&=\frac{1}{2}\Big(\E_{\xi,\bw}\left[\tfrac{d}{\rho}F(\bx+s\bm{\Delta}+\rho\bw;\xi)\bw\,|\,\bx,s,\bm{\Delta}\right]
\\
&~~~~~~~~~~+\E_{\xi,\bw}\left[\tfrac{d}{\rho}F(\bx+s\bm{\Delta}+\rho(-\bw);\xi)(-\bw)\,|\,\bx,s,\bm{\Delta}\right]\Big)
\\
&=\E_{\xi,\bw}\left[\tfrac{d}{\rho}F(\bx+s\bm{\Delta}+\rho\bw;\xi)\bw\,|\,\bx,s,\bm{\Delta}\right]~.
\end{align*}
Using the law of total expectation, we get
\begin{align*}
\E_{\xi,\bw}[\bg|\bx,s,\bm{\Delta}]
&=\E_{\bw}\left[\tfrac{d}{\rho}\E_{\xi}\left[F(\bx+s\bm{\Delta}+\rho\bw;\xi)\bw\,|\,\bw,\bx,s,\bm{\Delta}\right]\,|\,\bx,s,\bm{\Delta}\right]
\\
&=\E_{\bw}\left[\tfrac{d}{\rho}f(\bx+s\bm{\Delta}+\rho\bw)\bw\,|\,\bx,s,\bm{\Delta}\right]
\\
&=\nabla f_{\rho}(\bx+s\bm{\Delta})~,
\end{align*}
where the last equality is due to \citet[Lemma 2.1]{flaxman2005online}.
The second moment bound follows from
\citet[Lemma 10]{shamir2017optimal}, with the explicit constant pointed out by
\citet[Lemma E.1]{lin2022gradient}.

\end{proof}

The following result of \citet{cutkosky2023optimal} provides a stochastic first-order nonsmooth nonconvex optimization method, whose convergence scales linearly with the second-moment of the gradient estimator.

\begin{theorem}[\citealp{cutkosky2023optimal}] \label{thm: cutkosky}
Let $\nu,\epsilon\in(0,1)$, and let $h:\reals^d\to\reals$ be an $L$-Lipschitz function such that $h(\bx_0)-\inf_{\bx}h(\bx)\leq\Delta_h$.
Suppose $\textsc{GradEstimator}(\bx,\xi)$ returns an unbiased gradient estimator of $\nabla h(\bx)$ whose second moment is bounded by $\sigma^2$. Then there exists $T=O\left(\frac{\sigma^2\Delta_h}{\nu\epsilon^3}\right)$ such that setting
$\eta
=\frac{\Delta_h}{\sigma^{2}T},
~D
=\left(\frac{(\nu)^{1/2}\Delta_h}{\sigma T}\right)^{2/3}$
and running Algorithm~\ref{alg: main} uses $T$ calls to $\textsc{GradEstimator}$ and satisfies 
\begin{itemize}
    \item $\bz_{(k-1)M+m}\in\B(\overline{\bx}_k,\nu)$ for all $m\in[M],k\in[K]$ (where $M,K,(\bz_t)_{t=1}^{T}$ are defined in the algorithm).
    \item $\E_{\bz_1,\dots,\bz_T}\left[\frac{1}{K}\sum_{k=1}^{K}\norm{\frac{1}{M}\sum_{m=1}^{M}\nabla h(\bz_{(k-1)M+m})}\right]\leq\epsilon$ .
\end{itemize}
In particular, its output  $\bx^{\out}\sim\Unif\{\overline{\bx}_1,\dots,\overline{\bx}_{K}\}$ satisfies
$\E\left[\norm{\bpar_{_{\nu}}f(\bx^{\out})}\right]
\leq\epsilon$.

\end{theorem}

We are now ready to complete the proof of \thmref{thm: upper}. 
By Lemma~\ref{lem: estimator}, $\textsc{GradEstimator}$ (Algorithm~\ref{alg: grad}) returns an unbiased estimator of $\nabla f_{\rho}$ whose second moment is bounded by $\sigma^2=O(dL_0^2)$, using two evaluations of $F(\,\cdot\,;\xi)$. Thus applying Theorem~\ref{thm: cutkosky} to $h=f_{\rho}$
ensures that Algorithm~\ref{alg: main} returns a $(\nu,\epsilon)$-stationary point of $f_{\rho}$, which by Lemma~\ref{lem: delta,eps of smooth} is a $(\nu+\rho,\epsilon)$-stationary point of $f$.
Recall that $\norm{f-f_{\rho}}_{\infty}\leq{\rho L_0}$ and $f(\bx_0)-\inf_{\bx}f(\bx)\leq\Delta$, thus $f_{\rho}(\bx_0)-\inf_{\bx}f_{\rho}(\bx)\leq\Delta+\rho L_0=:\Delta_h$. 
Overall we have obtained a $(\nu+\rho)$-stationary point of $f$ using $2T$ function evaluations, where
\[
T=O\left(\frac{\sigma^2\Delta_h}{\nu\epsilon^3}\right)
=O\left(\frac{dL_0^2(\Delta+\rho L_0)}{\nu\epsilon^3}\right)
~.
\]
Setting $\rho=\min\left\{\frac{\delta}{2},\frac{\Delta}{L_0}\right\}$ and $\nu=\delta-\rho=\max\left\{\frac{\delta}{2},\delta-\frac{\Delta}{L_0}\right\}$ completes the proof, since $\nu+\rho=\delta$ and
\[
T=O\left(\frac{dL_0^2(\Delta+\min\left\{\frac{\delta}{2},\frac{\Delta}{L_0}\right\} L_0)}{\max\left\{\frac{\delta}{2},\delta-\frac{\Delta}{L_0}\right\}\epsilon^3}\right)
=O\left(\frac{dL_0^2(\Delta+\frac{\Delta}{L_0}\cdot L_0)}{\frac{\delta}{2}\cdot\epsilon^3}\right)
=O\left(\frac{dL_0^2\Delta}{\delta\epsilon^3}\right)~.
\]

\section{Proof of Theorem~\ref{thm: high prob}}

Recall that we saw in the proof of Theorem~\ref{thm: upper} that
$\partial_{\nu}f_{\rho}(\cdot)\subseteq \partial_{\nu+\rho}f(\cdot)$ according to Lemma~\ref{lem: delta,eps of smooth}, and
that for all $m\in[M]:\bz_m^{r^*}\in\B(\bx_\out^{r^*},\nu)$ by according to Theorem~\ref{thm: cutkosky}.
Thus 
\[
\norm{\bpar_{\delta}f(\bx_\out^{r^*})}
=\norm{\bpar_{\nu+\rho}f(\bx_\out^{r^*})}
\leq
\norm{\bpar_{\nu}f_{\rho}(\bx_\out^{r^*})}
\leq \norm{\frac{1}{M}\sum_{m\in[M]}\nabla f_{\rho}(\bz_m^{r^*})}
~.
\]
By denoting $\bg^{r}:=\frac{1}{M}\sum_{m\in[M]}\nabla f_{\rho}(\bz_m^{r}),~r\in[R]$ we get 
that it suffices to show that
\begin{equation} \label{eq: g<eps seeked}
\Pr\left[\norm{\bg^{r^*}}^2\leq\epsilon^2\right]
=\Pr\left[\norm{\bg^{r^*}}\leq\epsilon\right]\geq1-\gamma~.
\end{equation}
By definition of $r^*$ we have
\begin{align*}
    \norm{\hat{\bg}^{r^*}}^2
    =\min_{r\in[R]}\norm{\hat{\bg}^r}^2
    \leq \min_{r\in[R]}\left(2\norm{\bg^r}^2+2\norm{\hat{\bg}^r-\bg^r}^2\right)
\leq2\left(\min_{r\in[R]}\norm{\bg^r}^2+\max_{r\in[R]}\norm{\hat{\bg}^r-\bg^r}^2\right)~, 
\end{align*}
thus
\begin{align}
    \norm{\bg^{r^*}}^2
    &\leq 2\norm{\hat{\bg}^{r^*}}^2+2\norm{\hat{\bg}^{r^*}-\bg^{r^*}}^2 \nonumber
    \\
&\leq 4\left(\min_{r\in[R]}\norm{\bg^r}^2+\max_{r\in[R]}\norm{\hat{\bg}^r-\bg^r}^2\right)+2\norm{\hat{\bg}^{r^*}-\bg^{r^*}}^2 \nonumber
    \\
&\leq4\cdot\min_{r\in[R]}\norm{\bg^r}^2+6\cdot\max_{r\in[R]}\norm{\hat{\bg}^r-\bg^r}^2~. \label{eq: g before prob}
\end{align}
We now turn to bound each of the summand above with high probability. First, by Theorem~\ref{thm: upper} we can set
$T=O\left(\frac{d\Delta L_0^2}{\delta\epsilon^3}\right)$ so that $\E[\norm{\bg^{r}}]\leq\frac{\epsilon}{8}$ for any $r\in[R]$, hence by Markov's inequality
\[
\Pr\left[4\cdot\min_{r\in R}\norm{\bg^r}^2>\frac{\epsilon^{2}}{4}\right]
=\Pr\left[\min_{r\in R}\norm{\bg^r}>\frac{\epsilon}{4}\right]
\leq\prod_{r\in[R]}\Pr\left[\norm{\bg^r}>\frac{\epsilon}{4}\right]
\leq 2^{-R}~.
\]
By setting $R\geq\lceil\log_2(2/\gamma)\rceil$ so that $2^{-R}\leq\frac{\gamma}{2}$, we conclude that
\begin{equation} \label{eq: summand 1}
    \Pr\left[4\cdot\min_{r\in R}\norm{\bg^r}^2>\frac{\epsilon^{2}}{4}\right]
    \leq \frac{\gamma}{2}~.
\end{equation}
For the second summand in \eqref{eq: g before prob}, note that for all $r\in[R]:$
\begin{align*}
    \E\left[\hat{\bg}^r\right]
    &=\E\left[\frac{1}{S}\sum_{s=1}^{S}\bg^{r}_{s}\right]
    =\frac{1}{S}\sum_{s=1}^{S}\left(\frac{1}{M}\sum_{m=1}^{M}\E\left[\bg^r_{m,s}\right]\right)
    \\
    &=\frac{1}{M}\sum_{m=1}^{M}\left(\frac{1}{S}\sum_{s=1}^{S}\E\left[\bg^r_{m,s}\right]\right)
    \overset{\mathrm{Lemma~\ref{lem: estimator}}}{=}\frac{1}{M}\sum_{m=1}^{M}\nabla f_{\rho}
    (\bz_m^{r})
    =\bg^{r}~,
\end{align*}
thus $\E[\hat{\bg}^r-\bg^{r}]=0$, and that it follows from the second claim in Lemma~\ref{lem: estimator} that for any $r\in[R],s\in[S]:$
\begin{align*}
    \E\left[\norm{\hat{\bg}^r_s-\bg^{r}}^2\right]
    \leq\frac{16\sqrt{2\pi}dL_{0}^2}{M}~.
\end{align*}
Noting that $(\bg^r_{1}-\bg^r),\dots,(\bg^r_{S}-\bg^r)$ are independent as they are functions of the independent samples $\xi_{m,1},\dots,\xi_{m,S},~m\in[M]$, we apply a simple tail bound for the norm of a sum of independent vectors (Lemma~\ref{lem: concentration} in the appendix) to get for any $\lambda>0:$
\begin{align*}
\Pr\left[\norm{\hat{\bg}^r-\bg^r}^2\geq \lambda\frac{16\sqrt{2\pi}dL_{0}^2}{MS}\right]
&=\Pr\left[\norm{\frac{1}{S}\sum_{s=1}^{S}(\bg^r_{s}-\bg^r)}^2\geq \lambda\frac{16\sqrt{2\pi}dL_{0}^2}{MS}\right]
\\
&=\Pr\left[\norm{\sum_{s=1}^{S}(\bg^r_{s}-\bg^r)}^2\geq \lambda S\cdot\frac{16\sqrt{2\pi}dL_{0}^2}{M}\right]
\leq\frac{1}{\lambda}~,
\end{align*}
hence by the union bound
\begin{align*}
    \Pr\left[\max_{r\in[R]}\norm{\hat{\bg}^r-\bg^r}^2\geq \lambda\frac{16\sqrt{2\pi}dL_{0}^2}{MS}\right]\leq\frac{R}{\lambda}~.
\end{align*}
Setting $\lambda:=\lceil\frac{2R}{\gamma}\rceil$ so that $\frac{R}{\lambda}\leq\frac{\gamma}{2}$, we see that $S\gtrsim \frac{dL_{0}^2\log(1/\gamma)}{M\epsilon^2\gamma}$ suffices for having
$\lambda\frac{16\sqrt{2\pi}dL_{0}^2}{MS}\leq \frac{\epsilon^2}{4}$, under which the inequality above shows that
\begin{equation} \label{eq: summand 2}
\Pr\left[\max_{r\in[R]}\norm{\hat{\bg}^r-\bg^r}^2\geq \frac{\epsilon^2}{4}\right]\leq\frac{\gamma}{2}~.
\end{equation}
By combining \eqref{eq: summand 1} and \eqref{eq: summand 2} and applying the union bound to get \eqref{eq: g before prob}, we have proved \eqref{eq: g<eps seeked} as required.
Finally, recalling that $\textsc{GradEstimator}$ (Algorithm~\ref{alg: grad}) requires 2 noisy function evaluations, it is clear that the total number of evaluations performed by Algorithm~\ref{alg: high prob} is bounded by
\[
2R\cdot \left(T+MS\right)
=O\left(\frac{d L_0^2\Delta\log(1/\gamma)}{\delta\epsilon^3}+\frac{d L_{0}^2\log^2(1/\gamma)}{\gamma\epsilon^2}\right)~.
\]

\subsection*{Acknowledgements}
The authors would like to thank the anonymous JMLR reviewers for their helpful suggestions, and in particular for pointing out a modified assignment of hyper-parameters which led to an improved result with respect to the Lipschitz constant.
This research is supported in part by European Research Council (ERC) grant 754705. GK is supported by an Azrieli Foundation graduate fellowship.

\bibliographystyle{plainnat}
\bibliography{bib}

\appendix

\section{Concentration Lemma}

\begin{lemma} \label{lem: concentration}
Let $X_1,\dots,X_N\in\reals^d$ be independent random vectors such that
for all $i\in[N]:\E[X_{i}]=\mathbf{0},~\E[\norm{X_{i}}^2]\leq \sigma_{i}^2$. Then $\E\left[\norm{\sum_{i=1}^{N}X_i}^2\right]\leq\sum_{i=1}^{N}\sigma_i^{2}$. In particular, for any $\lambda>0:$
\[
\Pr\left[\norm{\sum_{i=1}^{N}X_i}^2\geq\lambda\cdot \sum_{i=1}^{N}\sigma_i^{2} \right]\leq\lambda^{-1}~.
\]
\end{lemma}

\begin{proof}
By linearity of expectation we have
    \[
    \E\left[\norm{\sum_{i\in[N]}X_i}^2\right]
    =\sum_{i\in[N]}\E\left[\norm{X_i}^2\right]
    +\sum_{i\neq j\in[N]}\E[\inner{X_i,X_j}]
    =\sum_{i\in[N]}\E\left[\norm{X_i}^2\right]
    \leq \sum_{i\in[N]}\sigma_i^2~,
    \]
    where we used the assumption that for any $i\neq j:X_i,X_j$ are independent, thus $\E[\inner{X_i,X_j}]=0$. The second claim follows from Markov's inequality.
\end{proof}

\end{document}